\documentclass[a4paper]{amsart}
\usepackage[margin=3.5cm]{geometry}
\usepackage[utf8]{inputenc}
\usepackage{amsmath,amsfonts,amssymb,amsthm}
\usepackage{mathtools,bbold}
\usepackage{tikz, tikz-cd}
\usepackage{comment, stackrel}

\theoremstyle{plain}
\newtheorem{theorem}{Theorem}[section]

\newtheorem{proposition}[theorem]{Proposition}

\newtheorem{corollary}[theorem]{Corollary}
\newtheorem{lemma}[theorem]{Lemma}

\theoremstyle{definition}

\newtheorem{example}[theorem]{Example}

\theoremstyle{remark}
\newtheorem{remark}[theorem]{Remark}

\newcommand{\SN}{\mathbb{N}}                    
\newcommand{\SC}{\mathbb{C}}                    
\newcommand{\CO}{\mathcal{O}}                    %
\newcommand{\CM}{\mathcal{M}}                    %
\newcommand{\kk}{\mathbb{k}}                    

\newcommand{\Rho}{\mathrm{P}}

\newcommand{\End}{\operatorname{End}}
\newcommand{\GL}{\operatorname{GL}}
\newcommand{\head}{\operatorname{h}}
\newcommand{\Hom}{\operatorname{Hom}}

\newcommand{\SL}{\operatorname{SL}}

\newcommand\adam{}
\newcommand\balazs{}

\DeclareMathOperator{\module}{\mathrm{-mod}}
\DeclareMathOperator{\rk}{\mathrm{rk}}

\mathchardef\mhyphen="2D

\newcommand\alg{A} 

\title{Slicing up multigraded linear series}

\author{\'Ad\'am Gyenge}
\address{Department of Algebra and Geometry, Institute of Mathematics, Budapest University of Technology and Economics, 
M\H{u}egyetem rakpart 3, H-1111, 
Budapest, Hungary}
\email{Gyenge.Adam@ttk.bme.hu}

\author{Bal\'{a}zs Szendr\H{o}i}
\address{University of Vienna, Austria}
\email{balazs.szendroi@univie.ac.at}


\begin{document}
\begin{abstract}
Multigraded linear series generalize the classical morphism to the linear series of a basepoint-free line bundle on a scheme. We investigate the collection of the natural cornering morphisms into elementary bigraded linear series obtained from direct summands of the original globally generated vector bundle. Our main result is a condition on the injectivity of the product morphism. We apply our result in two examples: modules over the reconstruction algebra and equivariant Hilbert and Quot schemes of quotient stacks. 
\end{abstract}

\maketitle



\section{Introduction}


Multigraded linear series were introduced by Craw-Ito-Karmazyn in \cite{craw2018multigraded} as a generalisation of linear series.
Given a scheme $Y$ equipped with a collection of globally generated vector bundles $E_1,\dots,E_n$ their construction provides a universal morphism from $Y$ to a certain fine moduli space $\CM (E)$ of cyclic modules over the endomorphism algebra of \[E = \CO \oplus E_1 \oplus \dots \oplus E_n.\] 

A natural construction arising in this context is {\balazs to restrict} the collection of vector bundles to a subset of the original ones:
\[ E_I= \CO \oplus \bigoplus_{i \in I} E_i \quad \textrm { for some } I \subseteq \{ 1,\dots,n \}. \]
 This process is called cornering, and there is a corresponding natural cornering morphism
\[ \CM (E) \to \CM (E_I) \]
between the original and the cornered moduli spaces.

A condition for the surjectivity of the cornering morphism was given in \cite[Proposition~3.7]{craw2018multigraded}. In this paper we consider the complementary question of injectivity. Of course a single cornering morphism is rarely injective (though see Remark~\ref{rem:inj} below). Therefore we consider a collection of cornering morphisms for subsets $I_1,\dots,I_l \subset\{0,\dots,n\}$.

In fact, our result applies in a more general setting involving certain {\adam quiver moduli spaces} $\CM (A,v)$ for a quiver algebra $A$ and dimension vector $v=(v_0,\dots,v_n)$. In the case of multigraded linear series the dimension vector has components $v_i=\mathrm{rk}\,E_i$ for each $i$. Such a quiver algebra $A$ {\balazs gives} rise similarly to cornered algebras $A_I$ which we consider together with the restricted dimension vector $v_I=(v_i)_{i\in I}$; for details see Section~\ref{sec:prelim}. Certain immersions of quiver varieties were investigated, with different methods, in the recent paper \cite{bertsch2025quiver}. Our main result is the following.
\begin{theorem}[{Theorem~\ref{thm:closedimm}}] 
\label{thm:main}
Let $I_1,\dots,I_l$ be subsets of $\{0,\dots,n\}$ that give a covering, that is, $\cup_{i=1}^l I_i=\{0,\dots,n\}$ such that $0 \in I_i$, $1 \leq i \leq l$. 
Then the product of the cornering morphisms 
\[\mathcal{M}(A,v) \to \prod_{i=1}^l \mathcal{M}(A_{I_l},v_{I_l})\]
induces a closed immersion on the reduced scheme underlying $\mathcal{M}(A,v)$.
\end{theorem}

The components of the image of a point of $\mathcal{M}(A,v)$ under this closed immersion can be considered as \lq slices\rq, from which the original point in $\mathcal{M}(A,v)$ can be reconstructed.

Such a situation occurs in a number of moduli problems. 
We apply first our result to modules over the reconstruction algebra associated with a finite subgroup $G < \GL(2,\SC)$ without pseudo-reflections. In this setting there exists an index set $\mathrm{Sp}$ and a bundle $T_{\mathrm{Sp}}$ on the corresponding $G$-Hilbert scheme such that $\mathcal{M}(T_{\mathrm{Sp}})$ is the minimal resolution of the quotient singularity $\SC^2/G$; the spaces $\mathcal{M}(T_{[\rho]})$ obtained from $\mathcal{M}(T_{\mathrm{Sp}})$ via cornering with $[\rho]=\{\rho_0, \rho\}$ give certain partial resolutions.
\begin{corollary} [{Corollary~\ref{cor:partresclimm}}]
\label{cor:1}
The product of the cornering morphisms
\[\mathcal{M}(T_{\mathrm{Sp}}) \to \prod_{\rho \in \mathrm{Sp} \setminus \rho_0} \mathcal{M}(T_{[\rho]})\]
is a closed immersion.
\end{corollary}

Second, when $\Gamma < \SL(2,\SC)$ is a finite subgroup {\balazs with $r+1$ conjugacy classes}, our result implies that a certain map from the equivariant Hilbert scheme of the quotient stack $[\SC^2/\Gamma]$ to the product of certain Quot schemes, which were constructed and identified with Nakajima quiver varieties in \cite{craw2021quot}, is a closed immersion.

\begin{corollary}[{Corollary~\ref{cor:hilbcorner}}] 
\label{cor:2}
Let $I_1,\dots,I_l$ be a covering of the set $\{0,\dots,r\}$ and $n \in \mathbb{N}^{r+1}$. The  product of the cornering morphisms
induces a closed immersion
\[ \mathrm{Hilb}^{n}([\SC^2/\Gamma]) \to \prod_{i=1}^l\mathrm{Quot}_{I_i}^{n_{I_i}} ([\SC^2/\Gamma]).\]
\end{corollary}
We remark that the relevant globally generated vector bundles on $\mathrm{Hilb}^{n}([\SC^2/\Gamma])$ correspond to vertices of a graph, the framed McKay quiver. These vertices are labelled $\{\infty,0,\dots,r\}$. The trivial bundle $\mathcal{O}$ corresponds to the vertex $\infty$, and it is implicitly included in each of the cornering sets; the schemes $\mathrm{Quot}_{I_i}^{n_{I_i}}([\SC^2/\Gamma])$ are the resulting cornered moduli spaces. Therefore, the setting of Corollary~\ref{cor:2} satisfies the assumptions of Theorem~\ref{thm:main}. The details will be explained in Section~\ref{sec:Hilb} below.

\begin{remark} \label{rem:inj} We discuss here a further example to illustrate our result, which however turns out to reduce to a classical statement. For integers $0<r<n$, let $E$ be the tilting bundle over the Grassmannian $\mathrm{Gr}(n, r)$ constructed in \cite{kapranov1984derived}. Its summands are indexed by Young diagrams of size at most $n-r$ and length at most $r$. Denote by $Y'(n-r,r)$ the set of such nontrivial diagrams. 
Let $[\lambda]=\{0,\lambda\}$ for any $\lambda \in Y'(n-r,r)$, where $0$ is the trivial partition $\lambda_1=\dots=\lambda_r=0$. By Theorem~\ref{thm:main}, there exists a closed immersion
\[\mathrm{Gr}(n, r) \to \prod_{\lambda \in Y'(n-r,r)}\mathcal{M}(E_{[\lambda]}).\]
However, one of the summands in the above bundle is
\[
\det(W)=S^{(1,\dots,1)}W.
\]
The corresponding morphism
\[
\mathrm{Gr}(n,r)\to \mathcal M(E_{[(1,\dots,1)]})
\]
is therefore the classical Pl\"ucker embedding. In particular, the product morphism
\[
\mathrm{Gr}(n,r)\to \prod_{\lambda\in Y'(n-r,r)} \mathcal M(E_{[\lambda]})
\]
is well known to be a closed immersion. 

More generally, if $Y=\mathcal M(E)$ and one summand $E_i$ is a very ample line bundle, then the universal morphism
\[
Y\to \mathcal{M}(E_{[i]})
\]
is a closed immersion, because the composition
with the inclusion $\mathcal{M}(E_{[i]}) \to |E_i|$ is the classical morphism to the linear series $|E_i|$. Hence, the induced product morphism to all cornered moduli spaces is also a closed immersion.
\end{remark}



The structure of the paper is as follows. In Section~\ref{sec:prelim} we collect necessarily preliminaries on multigraded linear series and their cornering. In Section~\ref{sec:reconstr} we establish our main technique for recovering modules from a set of its restrictions and we prove the main Theorem~\ref{thm:closedimm}.  Finally, in Section~\ref{sec:appl} we discuss the applications of the main result, proving Corollaries~\ref{cor:1} and \ref{cor:2}. 



\subsection*{Acknowledgements.} The first author was supported by the János Bolyai Research Scholarship of the Hungarian Academy of Sciences and by the National Research, Development and Innovation Fund of Hungary, within the Program of Excellence TKP2021-NVA-02 at the Budapest University of Technology and Economics. The authors thank the anonymous referee for numerous helpful remarks and corrections.

\section{Preliminaries}
\label{sec:prelim}

We work over an algebraically closed field $\kk$ of characteristic zero.

\subsection{Quiver varieties}

Throughout the paper we assume that $A$ is a finitely generated associative algebra which admits a presentation of the form
\[A \simeq \kk Q/I\]
for some finite connected quiver $Q$ and a two-sided ideal $I \subset \kk Q$ {\adam generated by linear combinations of paths of length at least one}. Denote the vertex set of $Q$ by $Q_0 = \{0, 1,\dots, n\}$ and by $(e_i)$, $i \in Q_0$ the corresponding set of orthogonal idempotents. 

Fix a dimension vector $v=(v_i) \in \SN^{Q_0}$. The space of stability conditions of $A$-modules of dimension $v$ is
\[ \Theta_{v}=\big\{\theta=(\theta_i): \mathbb{Z}^{Q_0} \to \mathbb{Q}\; : \; \sum_{i \in Q_0} \theta_iv_i=0 \big\}.\]
For $\theta \in \Theta_{v}$, an $A$-module $M$ of dimension vector $v$ is semistable (resp. stable) if $\theta(N) \geq 0$ (resp. $\theta(N) > 0$) for every nonzero proper $A$-submodule $N \subset M$. 
The space $\Theta_v$ admits a preorder $\geq$, and this induces a {\balazs finite} wall-and-chamber structure, where
$\theta,\theta' \in \Theta_v$ lie in the relative interior of the same cone if and only if both $\theta \geq \theta'$ and $\theta' \geq\theta$ hold in this preorder. The interiors of the top-dimensional cones in $\Theta_v$ are chambers, while the codimension-one faces of the closure of each chamber are walls. We say that $\theta \in \Theta_v$ is generic with respect to $v$, if it lies in some chamber. 

When the dimension vector $(v)$ is indivisible and $\theta \in \Theta_v$ is generic, there exists by \cite[Proposition 5.3]{king1994moduli} a fine moduli space
\[ \CM(A,v) \coloneqq \CM(A,v,\theta) \]
of isomorphism classes of $\theta$-stable $A$-modules of dimension vector $v$. There exists a universal family on $\CM(A,v,\theta)$, that is, a tautological locally-free sheaf 
\[T = \oplus_{i \in Q_0}T_i\] 
together with a $\kk$-algebra homomorphism $A \to \End(T)$ such that $\rk(T_i)=v_i$ for $0 \leq i \leq n$. 

In fact, the sheaf $T$ is defined only up to tensor product with an invertible sheaf \cite[Section~2]{craw2018multigraded}. But this ambiguity can be removed by choosing once and for all a vertex of the
quiver that we denote $0 \in Q_0$ and working only with dimension vectors $v$ satisfying $v_0= 1$; then
we normalise $T$ by fixing $T_0$ to be the trivial line bundle. 

With this choice, an $A$-module $M$ is said to be \emph{0-generated}, if there exists a surjective $A$-module homomorphism $Ae_0 \to M$. It follows from the definition that for stability parameters $\theta$ such that $\theta_i > 0$ for $i\neq 0$ all $\theta$-stable $A$-modules are 0-generated. We will use such parameters in our paper.

\subsection{Cornering}

We now obtain new algebras from $A$ as in \cite[Section~3]{craw2018multigraded}. Let 
\[I \subseteq \{0, 1,\dots, n\}\] be any nonempty subset. 
Define the
idempotent 
\[e_I \coloneqq \sum_{i\in I} e_i\] of $A$ and the $\kk$-algebra 
\[A_I \coloneqq e_I Ae_I.\] 
If the dimension vector $v_I = (v_i)_{i \in I}$ is also indivisible, we get a quiver variety \[\mathcal{M}(A_I,v_I) \coloneqq \mathcal{M}(A_I, v_I, \theta_I)\] where $\theta_I$ is a suitable stability parameter from $\Theta_{v_I}$.
The process of passing from $A$ to $A_I$ is called \emph{cornering} the algebra $A$. As we want to keep track of 0-generatedness, we only consider subsets $I$ that contain 0.

By \cite[Section~3]{craw2018multigraded}, for each pair of subsets 
\[\{0\} \subseteq I' \subseteq I \subseteq \{0, 1,\dots, n\}\]
and dimension vector $v_I \in \SN^I$ such that $v_I$ and $v_{I'}$ are both indivisible, 
there exists a universal \emph{cornering morphism} 
\[ \tau_{I,I'}: \mathcal{M}(A_I,v_I) \rightarrow \mathcal{M}(A_{I'},v_{I'}).\]


To simplify, we will write
\[ [i]=\{0,i\},\quad A_{[i]}=(e_0+e_i)A(e_0+e_i)
\]
and
\[ \tau_{[i]}=\tau_{\{0,\dots,n\},[i]}: \mathcal{M}(A,v) \to \mathcal{M}(A_{[i]},v_{[i]}) \]
for any $1 \leq i \leq n$.

\subsection{Multigraded linear series}
A particular case of the above setting is the following. Let $R$ be a finitely generated $\kk$-algebra and $Y$ be a projective $R$-scheme. Given
a collection $E_1,\dots, E_n$ of effective vector bundles on $Y$, define 
\[E \coloneqq \bigoplus_{0 \leq i \leq n} E_i\] 
where $E_0$ is the trivial line bundle on $Y$. Let 
\[A \coloneqq \End_Y (E)\] denote the endomorphism algebra of $E$ and
consider the dimension vector 
\[v \coloneqq (v_i), \quad v_i \coloneqq \rk(E_i),\quad 0 \leq i \leq n.\] 

It was observed in \cite[Proposition~2.3]{craw2018multigraded}, that $E$ is a flat family of 0-generated $A$-modules if and only if $E_i$ is globally generated for $1 \leq i \leq n$. In the examples in our paper we assume this and, as before, that $E_0$ is the trivial line bundle. The dimension vector 
\[v=(\rk(E_i))_{0 \leq i \leq n}\] is then indivisible. 

Choose a  stability parameter $\theta'$ such that \[\theta'_i = 1 \quad \textrm{for } i > 0.\]
Then $\theta'$ generic, and there is a smallest integer $j>0$ such that the ample line bundle $L_{\theta'}$
 provided by the GIT construction of $\CM(A,v,\theta')$ is very ample. Finally, set 
 \[\theta \coloneqq j \theta',\] which is again generic.
The \emph{multigraded linear series} of $E$ is the fine moduli space 
\[\CM(E)\coloneqq \CM(A,v,\theta)\]  of 0-generated $A$-modules
of dimension vector $v$ (see \cite[Definition~2.5]{craw2018multigraded}). 

Fix a nonempty subset $I \subseteq \{0,\dots,n\}$.
Since $A \simeq \End(E)$, the
locally-free sheaf 
\[E_I \coloneqq \bigoplus_{i \in I} E_i\] on $Y$ satisfies
\[A_I \simeq \End(E_I).\]
The corresponding multigraded linear series is \[\mathcal{M}(E_I)=\mathcal{M}(E_I, v_I, \theta_I),\] where $v_I = (\rk(E_i))_{i \in I}$ and $\theta_I$ is a suitable 0-generated stability parameter from $\Theta_{v_I}$. Again, for each pair of subsets 
\[\{0\} \subseteq I' \subseteq I \subseteq \{0, 1,\dots, n\}\]
there exists a universal cornering morphism 
\[ \tau_{I,I'}: \mathcal{M}(E_I) \rightarrow \mathcal{M}(E_{I'}).\]



Again, we will simply write for any $1 \leq i \leq n$ 
\[ [i]=\{0,i\},\quad E_{[i]}=E_0\oplus E_i 
\]
and
\[ \tau_{[i]}=\tau_{\{0,\dots,n\},[i]}: \mathcal{M}(E) \to \mathcal{M}(E_{[i]}) \]
for the cornering morphism.



\section{Slicing up $A$-modules}
\label{sec:reconstr}

\subsection{Recollement}

Denote by $A$-mod the category of all finitely presented $A$-modules.
Consider for $1 \leq i \leq n$ the functors
  \[
\begin{tikzpicture}
\node (v1) at (0,0)  {$\scriptsize{\alg\module}$};
\node (v2) at (4,0)  {$\scriptsize{\alg_{[i]}\module}$};

\draw [->] (v1) to node[above] {$\scriptstyle{j_{[i]}^{\ast}}$} (v2);
\draw [->,bend right=20] (v2) to node[above] {$\scriptstyle{j_{[i]\ast}}$} (v1);
\draw [->,bend left=20] (v2) to node[below] {$\scriptstyle{j_{[i]!}}$} (v1);
\end{tikzpicture}
\]
defined by $j_{[i]\ast}(-)\coloneqq \mathrm{Hom}_{\alg_{[i]}}(\alg,-)$, $j_{[i]}^{\ast}(-)\coloneqq (e_{0}+e_i)\alg\otimes_{\alg} (-)$ and $j_{[i]!}(-)\coloneqq \alg(e_{0}+e_i)\otimes_{\alg_{[i]}} (-)$. These are three of the six functors in a recollement of the category $\alg\module$ \cite{craw2018multigraded}. 


We need the following technical result about this adjunction setup.

\begin{lemma} \label{lem_universal_morph}
\begin{enumerate} 
\item Let $M\in\alg_{[i]}\module$. There is a canonical $\alg$-module homomorphism 
\[ \nu(M)\colon j_{[i]!} M \to j_{[i]\ast} M,\]
which is natural in $M$. 
\item Let $F\in \alg\module$, and $M=j_{[i]}^\ast F$. The homomorphism $\nu({j_{[i]}^\ast F})$ from (1) agrees with the composition
\[ j_{[i]!} j_{[i]}^\ast F\to F \to j_{[i]\ast} j_{[i]}^\ast F
\]
of the natural homomorphisms arising from the adjunctions.
\end{enumerate}
\end{lemma}

\begin{proof} Given $M\in\alg_{[i]}\module$, to show (1) we need to define a homomorphism of $\alg$-modules
\[ \nu_M\colon \alg(e_{0}+e_i)\otimes_{\alg_{[i]}} M \to \mathrm{Hom}_{\alg_{[i]}}(\alg,M).
\]
We define $\nu_M$ for $p\in \alg, m\in M$ and $a\in \alg$ by 
\begin{equation} p(e_{0}+e_i) \otimes m \mapsto \left( a\mapsto ((e_{0}+e_i)  a p (e_{0}+e_i)) \cdot m\right),
\label{eq:def:adjhom}
\end{equation}
extending this to sums of tensors linearly. 
This formula defines a homomorphism of left $\alg$-modules: the $\alg$-action on the left is multiplication on $p$ on the left; the $\alg$-action on the right results in multiplication on $a$ {\em on the right}. These clearly give the same result. It is also clear that the resulting homomorphism is natural in $M$.

For (2), we need to check that for $M=j_{[i]}^\ast F$, the homomorphism $\nu_{j_{[i]}^\ast F}$ agrees with the homomorphism coming from the adjunctions. The latter homomorphisms are
\[\alg(e_{{0}} + e_{i}) \otimes_{\alg_{[i]}} (e_{{0}}+e_i)\alg\otimes F\to F\]
given by
\[p(e_{0}+e_i) \otimes (e_{0}+e_i)\otimes f \mapsto p(e_{0}+e_i) \cdot f \]

and
\[F\to \mathrm{Hom}_{\alg_{[i]}}(\alg,(e_{0}+e_i)\alg\otimes_{\alg} F)
\]
given by
\[p(e_{0}+e_i) \cdot f \mapsto (a\mapsto (e_{0}+e_i)ap(e_{0}+e_i) \cdot f). \]
The composite of these two clearly agrees with~\eqref{eq:def:adjhom}. 
\end{proof}

For $F\in \alg\module$, combining for different $i\in I$ the homomorphisms from the adjunctions gives rise to canonical $\alg$-module homomorphisms
\[ \psi_F: \bigoplus_{i\in I}j_{[i]!} (j_{[i]}^{\ast}(F)) \to F\]
and
\[ \phi_F: F \to \bigoplus_{i\in I}j_{[i]\ast} (j_{[i]}^{\ast}(F)), \]
which are natural in $F$.
The following proposition will be crucial.

\begin{proposition}\label{prop:PiJembeds} 
 For any $F \in \alg\module$, the homomorphism $\psi_F$ is surjective, while the homomorphism $\phi_F$ is injective. In particular, 
\[ F\cong \mathrm{im} (\phi_F\circ \psi_F),\]
the image of the natural homomorphism $\phi_F\circ \psi_F$.
\end{proposition}

\begin{proof} To show the surjectivity of $\psi_F$, let us begin with the case $F=\alg\in\alg\module$. Write 
\begin{equation} 
\label{eq:Psidecomp}\widetilde{\alg}_{[i]}\coloneqq\alg(e_{{0}} + e_{i}) \otimes_{\alg_{[i]}} (e_{{0}}+e_i) \alg=j_{[i]!}(j_{[i]}^{\ast} (\alg))\end{equation} to simplify notation. 
Then the canonical left $\alg$-module homomorphism 
\[ \Psi\coloneqq\psi_{\alg}: \bigoplus_{i \in I} j_{[i]!}(j_{[i]}^{\ast} (\alg))= \bigoplus_{i \in I} \widetilde{\alg}_{[i]} \to \alg\]
can be written in the following explicit form. Any element of the source module is a sum of terms of the form
\[ (p_{m_1} \otimes q_{m_1},\dots, p_{m_k}\otimes q_{m_k}) \in \bigoplus_{i \in I} \widetilde{\alg}_{[i]}. \]
Then $\Psi$ maps such an element to $\sum_{s} p_{m_s}q_{m_s} \in \alg$ and extends this linearly.


For any $a\in \alg$, 
\[
\begin{gathered}
(p_{m_1} \otimes q_{m_1},\dots,p_{m_{k}}\otimes q_{m_{k}})a=(p_{m_1} \otimes q_{m_1}a,\dots,p_{m_{k}}\otimes q_{m_{k}}a) \\ \mapsto\sum_{s} p_{m_s}q_{m_s}a= \left(\sum_{s} p_{m_s}q_{m_s}\right)a .
\end{gathered}
\]
So $\Psi$ is in fact a bimodule homomorphism. 


To show the surjectivity of $\Psi$, consider the direct sum decomposition of the regular left module $_{\alg}{\alg}=e_{{0}}\alg  \oplus \bigoplus_{i \in I} e_i \alg$. Combine
\[
\begin{array}{ccc}
e_j\alg & \to & \bigoplus_{i \in I} \widetilde{\alg}_{[i]} \\
p & \mapsto & (0,\dots,0,(e_{0}+e_j)\otimes p,0,\dots,0)
\end{array}
\]
where $(e_{0}+e_j)\otimes p$ is at the $j$-th term of the target,
with
\[
\begin{array}{ccc}
e_{{0}} \alg & \to & \bigoplus_{i \in I} \widetilde{\alg}_{[i]} \\
p & \mapsto & \frac{1}{k}\left( (e_{0}+e_{m_1})\otimes p,\dots, (e_{0}+e_{m_{k}})\otimes p\right)
\end{array}\]
to obtain a linear map  \[ \Rho: \alg \to \bigoplus_{i \in I} \widetilde{\alg}_{[i]}.\] 
For any $a \in \alg$ and $p \in e_j\alg$, the product $pa$ is also in $e_j\alg$. Hence,
\[ \Rho(pa) =  (0,\dots,0,(e_{0}+e_j) \otimes pa,0,\dots,0) = (0,\dots,0,(e_{0}+e_j) \otimes p,0,\dots,0)a = \Rho(p)a.\]
Similarly, for $p \in e_{0}\alg$
\begin{align*} \Rho(pa) & = \frac{1}{k}\left( (e_{0}+e_{m_1})\otimes pa,\dots, (e_{0}+e_{m_{k}})\otimes pa\right) \\ & = \frac{1}{k}\left( (e_{0}+e_{m_1})\otimes p,\dots, (e_{0}+e_{m_{k}})\otimes p\right)a = \Rho(p)a. \end{align*}
Therefore $\Rho$ is a right $\alg$-module homomorphism. (We remark that it is not a $\alg$-bimodule homomorphism.)
It is straightforward to check that $\Psi\, \circ\, \Rho $ is the identity of $\alg$. We deduce that $\Psi$ is surjective as stated.

Next, for an arbitrary $F\in \alg\module$, we have 
\[ \bigoplus_{i \in I} j_{[i]!}(j_{[i]}^{\ast} (F))= \bigoplus_{i \in I} \alg(e_{{0}} + e_{i}) \otimes_{\alg_{[i]}} (e_{{0}}+e_i) \alg \otimes_{\alg} F= \bigoplus_{i \in I} \widetilde{\alg}_{[i]} \otimes_{\alg} F\]
and
\[ \psi_F = \Psi \otimes \mathrm{Id}_F: \bigoplus_{i \in I} \widetilde{\alg}_{[i]} \otimes_{\alg} F \to \alg\otimes_{\alg} F \cong F. \]
As the functor $-\otimes_{\alg} F$ is right exact, 
we deduce the surjectivity of $\psi_F$ for a general $F$.


Let us now consider $\phi_F$. 
The composition in the target can be written as
\[ \bigoplus_{i\in I}j_{[i]\ast} (j_{[i]}^{\ast}(F)) = \bigoplus_{i \in I} \mathrm{Hom}_{\alg_{[i]}}(\alg, (e_{{0}}+e_i)\alg \otimes_{\alg} F). \]
In this language, the map $\phi_F$ is the map
\[ 
\begin{array}{rc c l}
\phi_F\colon & F & \to &\bigoplus_{i \in I} \mathrm{Hom}_{\alg_{[i]}}(\alg, (e_{{0}}+e_i)\alg \otimes_{\alg} F)\\
& f &\mapsto &( a\mapsto (e_{{0}} + e_i)a \otimes f )_{i \in I}.
\end{array}
\]


Let $f \in F$ be such that \[ \phi_F(f)=( a\mapsto (e_{{0}} + e_i)a \otimes f )_{i \in I}=(0)_{i \in I}.\]
Take $a=e_j$ for $j \in I \cup \{{0}\}$ to see that in this case $e_j \otimes f=0$ for each $j$. But because $\{e_{j}\}_{j \in I \cup \{{0}\}}$ is a complete set of orthogonal idempotents of $\alg$, the sum
$\sum_j e_j \otimes f=1\otimes f=0$. This can only happen if $f=0$.
Thus $\phi_F$ is injective as claimed.
\end{proof}

\subsection{Closed immersion into the product}

We now apply Proposition~\ref{prop:PiJembeds} to prove our key result about multigraded linear series. 
The projective morphisms
\[ \tau_{[i]}\colon\mathcal{M}(A,v) \rightarrow \mathcal{M}(A_{[i]}, v_{[i]})\]
induce, by the universal property of the product, a projective morphism
\[ \tilde\tau\colon \mathcal{M}(A,v) \rightarrow\prod_{1 \leq i \leq n}\mathcal{M}(A_{[i]}, v_{[i]}).\]
\begin{theorem}
\label{thm:closedimm}
The morphism $\tilde\tau$
induces a closed immersion on the reduced scheme underlying $\mathcal{M}(A,v)$.
\end{theorem}
\begin{proof}


Let 
\[X=\mathop{\rm im}\tilde\tau \subset\prod_{1 \leq i \leq n}\mathcal{M}(A_{[i]}, v_{[i]})\] denote
the reduced image of $\tilde\tau$. This is a closed subvariety as $\tilde\tau$ is projective. We will show that $X$ is indeed isomorphic to the reduced scheme underlying $\mathcal{M}(A,v)$.


For $i\in I$, consider the universal bundle \[T_{{0},i} \oplus T_i\] on the fine moduli space 
$\mathcal{M}(A_{[i]}, v_{[i]})$, a sheaf of 
0-generated 
$\alg_{[i]}$-modules. Applying Lemma~\ref{lem_universal_morph}~(1) to this family, we
get a morphism
\[ \tilde\nu_i\colon j_{[i]!} (T_{{0},i} \oplus T_i) \to j_{[i]\ast} (T_{{0},i} \oplus T_i),
\]
of sheaves of $\alg$-modules on $\mathcal{M}(A_{[i]}, v_{[i]})$.

Pulling these bundles and morphisms back to the product 
\[\prod_{1 \leq i \leq n}\mathcal{M}(A_{[i]}, v_{[i]})\] 
and taking a direct sum, we get a morphism
\[ \widetilde\nu \colon \bigoplus_{1 \leq i \leq n} j_{[i]!} (T_{{0},i} \oplus T_i) \to \bigoplus_{1 \leq i \leq n} j_{[i]\ast} (T_{{0},i} \oplus T_i)
\]
of $\alg$-module bundles over $\prod_{1 \leq i \leq n}\mathcal{M}(A_{[i]}, v_{[i]})$, where we omit the pullback from the notation. Finally restrict this bundle morphism to the closed subscheme
\[X\subset \prod_{1 \leq i \leq n}\mathcal{M}(A_{[i]}, v_{[i]})\]
to get
\[ \widetilde\nu|_{X} \colon \bigoplus_{1 \leq i \leq n} j_{[i]!} (T_{{0},i} \oplus T_i)|_{X} \to \bigoplus_{1 \leq i \leq n} j_{[i]\ast} (T_{{0},i} \oplus T_i)|_{X}.
\]
{\balazs Consider the restriction 
\[S = \bigoplus_{1 \leq i \leq n} (T_{{0},i} \oplus T_i)|_{X}\]
of the direct sum bundle to $X=\mathop{\rm im}\tilde\tau$. 
Over a closed point $p\in X$, the fibre $S_p$ consists of a collection $\{F_{pi}\}$ of 0-generated 
$\alg_{[i]}$-modules of dimension vector $(1,v_i)$. Because $p$ is in the image of $\tilde\tau$, there exists at least one (not necessarily unique) 0-generated  
$\alg$-module $F_p$ of 
dimension vector $v$, such that $\tau_{[i]}(F_p) = F_{pi}$ for all $i$. 
On the other hand, if such an $F_p$ exists, then by  Lemma~\ref{lem_universal_morph}~(2) and Proposition~\ref{prop:PiJembeds}, at $p\in X$ the morphism $\widetilde\nu|_{p}$ has image exactly $F_p$. In particular, $F_p$ is the unique 0-generated 
$\alg$-module with the property that $\tau_{[i]}(F_p) = F_{pi}$ for all $i$. It follows that the image sheaf ${\rm Im}\ \widetilde\nu|_{X}$ is a family of 0-generated 
$\alg$-modules of dimension vector $v$.}
By the universal property of $\mathcal{M}(A,v)$, we then obtain a morphism \[\epsilon\colon X\to \mathcal{M}(A,v)\] which induces this family. 
By the uniqueness part of the argument just given,
$\epsilon\circ\tilde\tau$ is the identity on closed points
of $\mathcal{M}(A,v)$. The map $\tilde\tau\circ\epsilon$ is the identity of closed points 
on~$X$ by construction. Hence indeed $X$ is isomorphic to the reduced subscheme underlying $\mathcal{M}(E)$ as claimed.
\end{proof}

More generally, let $I_1,\dots,I_l$ be a covering of the set $\{0,\dots,n\}$ such that each $I_i$ contains 0. Again, the  morphisms
\[ \tau_{\{0,\dots,n\},I_i}: \mathcal{M}(A,v) \to \mathcal{M}(A_{I_i},v_{I_i})\]
induce a projective morphism
\[ \tau :\mathcal{M}(A,v) \to \prod_{i=1}^l \mathcal{M}(A_{I_i},v_{I_i}). \]
\begin{corollary}
\label{cor:corneredquiver}
The morphism $\tau$ induces a closed immersion on the reduced scheme underlying $\mathcal{M}(A,v)$.
\end{corollary}
\begin{proof}
{\adam For each $1 \leq j \leq n$, choose an index $i(j)\in \{1,\dots,l\}$ such that $j\in I_{i(j)}$, which is possible since $\{I_1,\dots,I_l\}$ covers $\{0,\dots,n\}$ and each $I_i$ contains $0$. Then the inclusion $[j]=\{0,j\}\subset I_{i(j)}$ induces a projective morphism
\[
\tau_{I_{i(j)},[j]} \colon \mathcal{M}(A_{I_{i(j)}},v_{I_{i(j)}})\to \mathcal{M}(A_{[j]},v_{[j]}).
\]
Taking the product over $j$ gives a morphism
\[
\rho \colon \prod_{i=1}^l \mathcal{M}(A_{I_i},v_{I_i})
\longrightarrow
\prod_{1\leq j\leq n}\mathcal{M}(A_{[j]},v_{[j]}),
\]
such that
\[
\tilde\tau=\rho\circ\tau,
\]
where $\tilde\tau$ is the morphism of Theorem~\ref{thm:closedimm}. Since $\tilde\tau$ induces a closed immersion on the reduced scheme underlying $\mathcal{M}(A,v)$, the same holds for $\tau$, because composition with $\rho$ cannot identify two distinct points in the image of $\tau$. Hence $\tau$ induces a closed immersion on the reduced scheme underlying $\mathcal{M}(A,v)$.}
\end{proof}


Specially, for a multigraded linear series $\mathcal{M}(E)$ the morphisms
\[ \tau_{\{0,\dots,n\},I_i}: \mathcal{M}(E) \to \mathcal{M}(E_{I_i})\]
give rise to a projective morphism
\[ \tau :\mathcal{M}(E) \to \prod_{i=1}^l \mathcal{M}(E_{I_l}). \]
\begin{corollary}
\label{cor:corneredmgls}
The morphism $\tau$ induces a closed immersion on the reduced scheme underlying $\mathcal{M}(E)$.
\end{corollary}

{\adam 
\begin{remark}
A special case of the closed immersion of Corollary \ref{cor:corneredmgls} have appeared in the literature for the case when $Y$ is a Mori Dream Space and each $E_i$ has rank one. For any collection of  globally generated
line bundles $\mathcal{L} = (\mathcal{O}_Y
, E_1, \dots , E_n)$, the image of the closed immersion $ \varphi_{|\mathcal{L}|}:\mathcal{M}(E) \to |\mathcal{L}|$ from \cite[Theorem~3.7]{craw2013mori} is cut out by the relations in the quiver. So the morphism \[\prod_{1 \leq i \leq n} \varphi_{|E_i|} : \mathcal{M}(E) \to \prod_{1 \leq i \leq n}|E_i|\] is a closed immersion with image in \[\prod_{1 \leq i \leq n} \mathcal{M}(O_Y \oplus E_i).\] This also provides a geometric intuition for our the closed immersion above.
\end{remark}

}

\section{Examples}
\label{sec:appl}

\subsection{Modules over the reconstruction algebra}

In this section we work over the complex numbers. Let $G < \GL(2,\SC)$ be a finite subgroup that acts without pseudo-reflections. Let $\mathrm{Irr}(G)$ denote the set of isomorphism classes of irreducible representations of $G$. There is a tautological bundle
\[ T=\bigoplus_{\rho \in \mathrm{Irr}(G)} T_{\rho}^{\oplus \mathrm{dim}(\rho)} \]
on the corresponding $G$-Hilbert scheme, the fine module space of $G$-equivariant coherent sheaves of the form $\mathcal{O}_Z$ for $Z \subset \mathbb{A}^2$ such that $\Gamma(\mathcal{O}_Z)$ is isomorphic to the regular representation of $G$. Each summand $T_{\rho}$ of this bundle is globally generated. In this context, the trivial representation corresponds to the trivial line bundle.

Recall that an irreducible representation $\rho \in \mathrm{Irr}(G)$ is called special if $H^1(T_{\rho}^{\vee})=0$. Performing cornering along the set
\[ \mathrm{Sp} \coloneqq \{ \rho \in  \mathrm{Irr}(G) \,\vert\, \rho \textrm{ is special }\}\]
one obtains the reconstruction bundle 
\[ T_{\mathrm{Sp}} \coloneqq {\balazs \bigoplus_{\rho \in \mathrm{Sp}}T_{\rho}}\]
and the reconstruction algebra
\[ B \coloneqq \End(T_{\mathrm{Sp}}) = \End(T)_{\mathrm{Sp}} = e\End(T)e\]
where $e = \sum_{\rho \in \mathrm{Sp}} e_{\rho}$.
\begin{theorem}[{\cite[Theorem~4.4]{craw2018multigraded}}]
\begin{enumerate}
    \item The minimal resolution of the quotient singularity $\mathbb{A}^2/G$ is isomorphic to the multigraded linear series $\mathcal{M}(T_{\mathrm{Sp}})$. In particular, $\mathcal{M}(T_{\mathrm{Sp}})$ is reduced.
    \item All partial resolutions of $\mathbb{A}^2/G$ which  the minimal resolution factors through is of the form  $\mathcal{M}(T_I)$ for some $I \subset \mathrm{Sp}$.
\end{enumerate}
\end{theorem}
It is known that the set of invertible sheaves 
\[ \{\det(T_{\rho})\,\vert\, \rho \in  \mathrm{Sp} \setminus \rho_0\}\]
provides an integral basis of $\mathrm{Pic}(\mathcal{M}(T_{\mathrm{Sp}}))$ dual to the curve classes defined by the irreducible  components of the exceptional divisor. The cornering morphism 
\[\mathcal{M}(T_{\mathrm{Sp}}) \to \mathcal{M}(T_{[\rho]})\]  contracts every exceptional curve except for the one corresponding to $\rho$.
In this setting, our Theorem~\ref{thm:closedimm} implies that
the product of the cornering morphisms gives a closed immersion.
\begin{corollary} 
\label{cor:partresclimm}
The product of the cornering morphisms
\[\mathcal{M}(T_{\mathrm{Sp}}) \to \prod_{\rho \in \mathrm{Sp} \setminus \rho_0} \mathcal{M}(T_{[\rho]})\]
is a closed immersion.
\end{corollary}

\begin{remark}
As the singularity $\mathbb{A}^2/\Gamma$ is rational, 
\[ h^1(\mathcal{M}(T_{[\rho]}),\mathcal{O}_{\mathcal{M}(T_{[\rho]})})= h^1(\mathcal{M}(T_{\mathrm{Sp}}),\mathcal{O}_{\mathcal{M}(T_{\mathrm{Sp}})})=0.\] Therefore, by \cite[Exercise~III.12.6.]{hartshorne1977algebraic}
\[ \mathrm{Pic}\left(\prod_{\rho \in \mathrm{Sp} \setminus \rho_0} \mathcal{M}(T_{[\rho]})\right) = \bigoplus_{\rho \in \mathrm{Sp} \setminus \rho_0} \mathrm{Pic} (\mathcal{M}(T_{[\rho]}))\simeq \mathrm{Pic}(\mathcal{M}(T_{\mathrm{Sp}})).\]
\end{remark}

\subsection{Hilbert schemes of Kleinian orbifolds}
\label{sec:Hilb}


Let $\Gamma < \SL(2,\SC)$ be a finite subgroup. The McKay graph of $\Gamma$ is defined to be the graph with vertex set $\{0,1, \dots, r\}$ corresponding to {\balazs the irreducible representations $\rho_0=\rm{triv}, \rho_1, \ldots, \rho_r$} of $\Gamma$ such that, for each $0\leq i, j\leq r$, we have $\dim\Hom_{\Gamma}(\rho_j , \rho_i \otimes V )$ edges joining vertices $i$ and $j$. McKay~\cite{McKay80} observed that this graph is an extended Dynkin diagram of ADE type. The framed McKay graph of $\Gamma$ is obtained by introducing an additional vertex, denoted $\infty$, together with an edge joining the vertices $\infty$ and $0$.

Associate the doubled quiver to the framed McKay graph. For this, consider the set of pairs $Q_1$ comprising an edge in the framed McKay graph and an orientation of the edge. For each $a\in Q_1$, and we write $t(a)$, $h(a)$ for the vertices at the tail and head respectively of the oriented edge, and we write $a^{\ast}$ for the same edge with the opposite orientation. The \emph{framed McKay quiver of $\Gamma$}, denoted $Q$, is the quiver with vertex set {\balazs $Q_0= \{\infty,0,\dots ,r\}$} and arrow set $Q_1$. 

Let $\SC Q$ denote the path algebra of the framed McKay quiver $Q$. For $i \in Q_0$, let $e_i \in  \SC Q$ denote the idempotent corresponding to the trivial path at vertex $i$. Let $\epsilon\colon Q_1 \to \{\pm 1\}$ be any map such that $\epsilon(a) \neq \epsilon(a^{\ast})$ for all $a \in Q_1$. The \emph{preprojective algebra of $Q$}, denoted $\Pi$, is defined as the quotient of $\SC Q$ by the ideal
\begin{equation}
\label{eqn:preprojectiverelation}
\left\langle \sum_{\head(a)=i} \epsilon(a) aa^{\ast} \mid {i \in Q_0}
\right\rangle.
\end{equation}
 The preprojective algebra $\Pi$ does not depend on the choice of the map $\epsilon$. For each $i\in Q_0$, we simply write $e_i\in \Pi$ for the image of the corresponding vertex idempotent.

Let $A=\Pi/(b^{\ast})$ be the quotient by the two sided ideal $(b^{\ast})$, where $b^\ast$ is the unique arrow from $0$ to $\infty$. In this context, the vertex $\infty$, and not $0$, will play the distinguished role.
Define the idempotent \[\overline{e}_I \coloneqq e_{\infty} + \sum_{i \in I} e_i \in A.\] 
{\balazs For a non-empty subset $I \subseteq \{0,\dots,r\}$,} the subalgebra
\begin{equation}
\label{eqn:AI}
A_I \coloneqq \overline{e}_I A \overline{e}_I 
\end{equation}
of $A$ comprises linear combinations of 
classes of paths in $Q$ whose tails and heads lie in the set $\{\infty\} \cup I$. Let \[m_I=(m_{i_1},\dots,m_{i_{|I|}}) \in \mathbb{N}^{I}\] be a dimension vector of size $|I|$.
In \cite[Section~3.2]{craw2021quot}, a Quot scheme was constructed that parameterizes $m_I$-(multi)dimensional locally-free quotients of certain sheaves on the Kleinian orbifold $[\mathbb{C}^2 / \Gamma]$. Denote this scheme by \[\mathrm{Quot}_I^{m_I}([\SC^2/\Gamma]).\] As a special case, for $I = \{0,\dots,r\}$ and any $m \in \mathbb{N}^{r+1}$
\[\mathrm{Quot}_I^{m} ([\SC^2/\Gamma]) \cong \mathrm{Hilb}^{m}([\SC^2/\Gamma])\]
where
\begin{equation} 
\label{eq:orbihilb}
\mathrm{Hilb}^{m}([\SC^2/\Gamma]) = \Big\{J\in \mathrm{Hilb}\big(\SC^2\big)^{\Gamma}\mid H^0(\mathcal{O}_{\SC^2}/J) \cong \bigoplus_{i\in \{0, \dots, n\}} \rho_i^{\oplus m_i}\Big\},\
\end{equation}
a component of the $\Gamma$-equivariant Hilbert scheme of the plane; see \cite[Lemma~3.3]{craw2021quot}. Here we suppressed the index set $I$ from $m$ because it equals with the whole set $\{0,\dots,r\}$.

Extend the vector $m_I$ as \[(1,m_I) \in \mathbb{N} \oplus \mathbb{N}^{I}\]  
to get a dimension vector for $A_I$-modules, and consider the stability condition $\eta_{I}\colon \mathbb{Z} \oplus \mathbb{Z}^{I} \to \mathbb{Q}$
given by
\[
\eta_I (\rho_i) =\left\{
\begin{array}{c r}
-\sum_{j \in I} m_j & \textrm{for } i= \infty \\
1 & \textrm{if }i \in I
\end{array}\right. 
\]
The vector $(1,m_I)$ is indivisible and $\eta_I$ is a generic stability condition for $A_I$-modules of dimension vector $(1,m_I)$, so there is a fine moduli space $\mathcal{M}_{A_I}(1,m_I)$
of $\eta_I$-stable $A_I$-modules of dimension
vector $(1,m_I)$. Let \[T_I \coloneqq\bigoplus_{i \in \{\infty \}\cup I } T_i\] denote the tautological vector bundle on $\mathcal{M}_{A_I}(1,m_I)$, where $T_{\infty}$ is the trivial bundle and $T_i$ has rank $m_i$ for $i \in I$. This setting is exactly the situation of Section~\ref{sec:prelim}, so our Corollary~\ref{cor:corneredquiver} applies.

\begin{proposition}[{\cite[Proposition~4.2]{craw2021quot}}]
For any non-empty subset $I\subseteq\{0,\dots,r\}$ and $m_I \in \mathbb{N}^{I}$, there is an isomorphism of schemes 
\[ \mathrm{Quot}_I^{m_I}([\SC^2/\Gamma]) \longrightarrow \mathcal{M}_{A_I}(1,m_I).\]
\end{proposition}

Therefore, for each nonempty $I\subset \{0,\dots,r\}$, there is a cornering morphism
\[ \mathrm{Hilb}^{m}([\SC^2/\Gamma]) \to \mathrm{Quot}_{I}^{m_{I}} ([\SC^2/\Gamma])\]
where $m_I=m|_{I}$, the restriction of $m$ to the indices contained in $I$. 
\begin{example}
For $I=\{0\}$ and $m_0 \in \mathbb{N}$,
\[\mathrm{Quot}_I^{m_0} ([\SC^2/\Gamma]) \cong \mathrm{Hilb}^{m_0}(\SC^2/\Gamma),\]
the Hilbert scheme of $m_0$ points on the quotient singularity $\SC^2/\Gamma$. Then for any $m \in \mathbb{N}^{n+1}$ the cornering morphism 
\[ \mathrm{Hilb}^{m}([\SC^2/\Gamma]) \to \mathrm{Hilb}^{m_0}(\SC^2/\Gamma) \]
can be identified with the map
\[ J \mapsto J^{\Gamma}\]
under the description \eqref{eq:orbihilb}.
This map is a resolution of singularities by \cite{craw2019punctual} if $m=m_0 \cdot \rho_{reg}$ with $\rho_{reg}$ the regular representation of $\Gamma$.

More generally, for a singleton $I=\{i\}$ with $i \in \{0,\dots,r\}$ the map 
\[ \mathrm{Hilb}^{m}([\SC^2/\Gamma]) \to \mathrm{Quot}_{i}^{m_{i}} ([\SC^2/\Gamma])\]
can be described as the projection
\[ J \to J^{\rho_i},\]
where $J^{\rho_i}$ is the $\rho_i$-isotypic part of $J$. 
When $\dim(\rho_i)=1$, there is yet another geometric description of the target as a Quot scheme of a certain sheaf on the quotient singularity $\SC^2/\Gamma$; for the details see \cite[Proposition~3.4]{craw2021quot}.
\end{example}

\begin{corollary}\label{cor:hilbcorner}
Let $I_1,\dots,I_l$ be a covering of the set $\{0,\dots,r\}$ and $m \in \mathbb{N}^{r+1}$. The  product of the cornering morphisms
induces a closed immersion
\[ \mathrm{Hilb}^{m}([\SC^2/\Gamma]) \to \prod_{i=1}^l\mathrm{Quot}_{I_i}^{m_{I_i}} ([\SC^2/\Gamma]) \]
\end{corollary}

\bibliographystyle{amsplain}
\bibliography{recoll}

\end{document}